%% file: Egenerated.tex
\documentclass[12pt, a4paper]{amsart}

\input{preamble}

\input{bold.greek}

\newcommand{\locus}{\mathrm{locus}}
\newcommand{\pr}{\mathrm{pr}}

\title[Existentially generated subfields of large fields]{\rm\large Existentially generated subfields of large fields}
\author{Sylvy Anscombe}
\thanks{\today}
\address{Jeremiah Horrocks Institute, University of Central Lancashire, Preston PR1 2HE, United Kingdom}
\email{sanscombe@uclan.ac.uk}

\begin{document}
\begin{abstract}
We study subfields of large fields which are generated by infinite existentially definable subsets.
We say that such subfields are {\bf existentially generated}.

Let $L$ be a large field of characteristic exponent $p$, and let $E\subseteq L$ be an infinite existentially generated subfield.
We show that $E$ contains $L^{(p^{n})}$, the $p^{n}$-th powers in $L$, for some $n<\omega$.
This generalises a result of Fehm
from \cite{Fehm10},
which shows $E=L$, under the assumption that $L$ is perfect.
Our method is to first study existentially generated subfields of henselian fields.
Since $L$ is existentially closed
in the henselian field $L((t))$,
our result follows.
\end{abstract}
\maketitle

Large fields were introduced in \cite{Pop96} by Pop:
A field $L$ is {\bf large}\footnote{Large fields are also known as {\em ample fields}.}
if every smooth curve defined over $L$ with at least one $L$-rational point has infinitely many $L$-rational points.
A survey of the theory of large fields is given in \cite{BF}.

Our fields have characteristic exponent $p$, i.e.\ $p$ is the characteristic, if this is positive, and otherwise $p=1$.
A subset $X\subseteq L$ is {\bf existentially definable} if it is defined by an existential formula from the language of rings, allowing parameters.
We denote by $(X)$ the subfield generated by $X$.
A subfield $E\subseteq L$ is {\bf existentially generated} if there is an infinite existentially definable subset $X\subseteq L$ which generates $E$, i.e.\ $E=(X)$.

In \autoref{section:proof} we prove the following theorem.

\begin{theorem}\label{thm:main.theorem}
Let $L$ be a large field of characteristic exponent $p$,
and let $E\subseteq L$ be an existentially generated subfield.
Then we have
$$L^{(p^{n})}\subseteq E,$$
for some $n<\omega$,
where $L^{(p^{n})}=\{x^{p^{n}}|x\in L\}$ is the subfield of $p^{n}$-th powers.
\end{theorem}

The motivation for this work was the following result of Arno Fehm.

\begin{theorem}[Corollary 9, \cite{Fehm10}]\label{thm:Fehm}
A perfect large field $L$ has no existentially $L$-definable proper infinite subfields.
\end{theorem}

In fact, using our terminology, Fehm's method immediately shows that
{\em a perfect large field $L$ has no existentially generated proper subfields}.
For imperfect $L$ and each $n<\omega$, the subfield $L^{(p^{n})}$ is existentially definable, without parameters, by using the Frobenius map.
Moreover, if we use parameters then we are able to existentially define various extensions of the subfields $L^{(p^{n})}$.
Thus, our result generalises \autoref{thm:Fehm} by removing the assumption that $L$ is perfect.
On the other hand, if the characteristic of $L$ is zero, then $L$ is necessarily perfect, so Fehm's result already applies.

The key to our method is to study the same problem in a henselian field $K$, i.e.\ a field equipped with a nontrivial henselian valuation.
First, we recall some facts about separable field extensions in \autoref{section:preliminaries}.
Then in the context of an arbitrary field, in \autoref{section:big} we introduce and study `{\em big} subfields'; and in \autoref{section:uniformly.big} we introduce and study `{\em uniformly big} subfields'.
In \autoref{section:Egen.subfields}
we show that existentially generated subfields of henselian fields are uniformly big,
and that they contain `sufficiently many' points of $K^{(p^{\infty})}$.
From this we can deduce \autoref{thm:main.theorem}, restricted to henselian fields.
Finally, 
in \autoref{section:proof},
we use the fact that $L$ is existentially closed in the henselian field $L((t))$ to finish the proof of \autoref{thm:main.theorem}.


\subsection*{Notation}

Throughout, $C,E,F,K,L$ will denote fields, $C$ will usually be a subfield `of parameters', $K$ will be henselian, and $L$ will be large.
To avoid confusion between Cartesian products and powers, for $n<\omega$ and a set $X$, we let $X^{n}=X\times\ldots\times X$ denote the $n$-fold Cartesian product, and let $X^{(n)}=\{x^{n}|x\in X\}$ denote the set of $n$-th powers of elements from $X$.
Sometimes it will be convenient to think of tuples as being indexed by a tuple of variables.
If $\mathbf{x}=(x_{1},...,x_{m})$ is an $m$-tuple of variables, we write $X^{\mathbf{x}}$ for the set of $m$-tuples from $X$ indexed by $\mathbf{x}$.

Let $\mathbf{x}=(x_{1},...,x_{m})$, $\mathbf{y}=(y_{1},...,y_{n})$ be two tuples of variables.
Despite the abuse of language, we say $\mathbf{x}$ is a {\bf subtuple} of $\mathbf{y}$ if $\{x_{1},...,x_{m}\}\subseteq\{y_{1},...,y_{n}\}$.
In this case, we write
$\pr_{\mathbf{x}}:X^{\mathbf{y}}\longrightarrow X^{\mathbf{x}}$
for the {\bf projection} that maps each $\mathbf{y}$-tuple to its subtuple corresponding to $\mathbf{x}$.

Let $F/C$ be any field extension and let $\mathbf{a}\in F^{m}$.
We define the {\bf locus} of $\mathbf{a}$ in $F$ over $C$ to be the following:
\begin{align*}
\locus(\mathbf{a}/C)&:=\left\{\mathbf{b}\in F^{m}\;\Big|\;\forall f\in C[X_{1},...,X_{m}]\;\big(f(\mathbf{a})=0\implies f(\mathbf{b})=0\big)\right\}.
\end{align*}
In other language, this is the set of $F$-rational points of the smallest Zariski-closed subset of affine $m$-space, which is defined over $C$ and contains $\mathbf{a}$.
In more algebro-geometric contexts, a different notion of `locus' is used.
For example,
in \cite[Chapter II, Section 3]{LangAG}, Lang defines the locus $\mathbf{a}$ over $C$ to be the set of all specializations of $\mathbf{a}$, which forms a variety. 
Our locus is simply the set of $F$-rational points of this variety.
In the present paper, we are more interested in $F$-rational points of varieties than in varieties themselves, so there should be no confusion.

Suppose that $\mathbf{x}$ is a subtuple of $\mathbf{y}$.
Let $\mathbf{b}\in F^{\mathbf{y}}$, and let $\mathbf{a}$ be the subtuple of $\mathbf{b}$ corresponding to the variables $\mathbf{x}$.
Restricting the projection map, from above, we obtain
\begin{align*}
\locus(\mathbf{b}/C)\longrightarrow\locus(\mathbf{a}/C),
\end{align*}
which we also denote by $\pr_{\mathbf{x}}$.
Note that this will not be surjective, in general.


\section{Relative inseparable closure}
\label{section:preliminaries}

It is well-understood that if $F/C$ is a finite normal field extension, then there is a unique intermediate field $D$, such that $F/D$ is separably algebraic and $D/C$ is purely inseparable.
In \cite{DeveneyMordeson}, Deveney and Mordeson study an analogous problem in which we do not assume $F/C$ to be algebraic.
Before stating their result, we first recall some basic definitions.

\begin{definition}
A tuple $\mathbf{a}\subseteq F$ is a {\bf separating transcendence base} of an extension $F/C$ if $\mathbf{a}$ is algebraically independent over $C$ and $F/C(\mathbf{a})$ is separably algebraic.
\end{definition}

The following, which for our purposes we treat as a definition, is sometimes known as {\em Mac Lane's Criterion}.
For other equivalent statements and a more detailed development of the subject of separability of field extensions, see \cite{ML} and \cite[Chapter VIII, Proposition 4.1]{Lang}.

\begin{definition}
A field extension $F/C$ is {\bf separable} if every finite tuple $\mathbf{a}\subseteq F$ contains a separating transcendence base of $C(\mathbf{a})/C$.
\end{definition}

Now we come to the theorem of Deveney and Mordeson, which we re-word for our convenience.

\begin{fact}[cf {\cite[Theorem 1.1]{DeveneyMordeson}}]\label{fact:Deveney.Mordeson}
Let $F/C$ be a field extension.
Consider the set
\begin{align*}
\mathcal{S}_{F}(C)&:=\{E\subseteq F\;|\;\text{$F/E$ is separable and $C\subseteq E$}\}.
\end{align*}
Then $\mathcal{S}_{F}(C)$ has a minimum element, with respect to inclusion.
That is, there is $D\in\mathcal{S}_{F}(C)$ such that for all $E\in\mathcal{S}_{F}(C)$ we have $D\subseteq E$.
\end{fact}

\begin{definition}\label{def:relative.inseparable.closure}
Let $F/C$ be a field extension.
The {\bf relative inseparable closure} of $C$ in $F$, which we denote by $\Lambda_{F}(C)$,
is the minimum element of $\mathcal{S}_{F}(C)$.
If a tuple $\mathbf{q}\subseteq F$ generates a subfield $C=(\mathbf{q})$ of $F$,
then we say the {\bf relative inseparable closure} of $\mathbf{q}$ in $F$ is
$\Lambda_{F}(\mathbf{q}):=\Lambda_{F}(C)$.
\end{definition}



\begin{lemma}\label{lem:absoluteness}
Let $F\preceq F^{*}$ be an elementary extension.
Then $\Lambda_{F}(C)=\Lambda_{F^{*}}(C)$.
\end{lemma}
\begin{proof}
This follows from \autoref{fact:Deveney.Mordeson}, and the fact that if $F/E$ is separable then $F^{*}/E$ is separable, for any $E$.
\end{proof}

\begin{lemma}\label{lem:countable}
Let $\mathbf{q}\in F^{n}$, for some $n<\omega$.
Then $|\Lambda_{F}(\mathbf{q})|\leq\aleph_{0}$.
\end{lemma}
\begin{proof}
There are several way of seeing this.
For example, let $F_{*}\preceq F$ be an elementary submodel with $\mathbf{q}\subseteq F_{*}$ and $|F_{*}|\leq\aleph_{0}$.
Of course it follows that $|\Lambda_{F_{*}}(\mathbf{q})|\leq\aleph_{0}$, and
in fact $\Lambda_{F}(\mathbf{q})=\Lambda_{F_{*}}(\mathbf{q})$,
by \autoref{lem:absoluteness}.
\end{proof}

For $n<\omega$, recall $F^{(p^{n})}$ is the subfield of $p^{n}$-th powers of elements of $F$.
We denote by
$$F^{(p^{\infty})}:=\bigcap_{n<\omega}F^{(p^{n})}$$
the largest perfect subfield of $F$.
In the proof of the next lemma, we use the notion of $p$-independence, introduced by Teichm\"{u}ller in \cite{T}, and later developed by Mac Lane, for example in \cite{ML}.
Since this lemma is the only place in the present paper that this notion appears, we do not give the definitions.

\begin{lemma}\label{lem:roots}
Let $F/C$ be separable and let $a\in F^{(p^{\infty})}$.
Then $\Lambda_{F}(C(a))=C(a^{p^{-n}}\;|\;n<\omega)$.
\end{lemma}
\begin{proof}
It is clear that $\Lambda_{F}(C(a))\supseteq C(a^{p^{-n}}\;|\;n<\omega)$.
For the reverse inclusion, by it suffices to show that some maximal $p$-independent tuple (i.e.\ a {\em $p$-base}) of $C(a^{p^{-n}}\;|\;n<\omega)$ is $p$-independent in $F$.
Let $\mathbf{c}\subseteq C$ be any $p$-base of $C$.
Then $\mathbf{c}$ is a $p$-base of $C(a^{p^{-n}}\;|\;n<\omega)$.
Since $F/C$ is separable, $\mathbf{c}$ is $p$-independent in $F$.
\end{proof}


\section{Big subfields}
\label{section:big}

In this section we introduce the {\em ad hoc} notion of `big subfields'.
Let $F/E$ be any field extension.
We show in \autoref{prp:big.subfields} that if $E$ is a big subfield of $F$ then $F^{(p^{\infty})}/E^{(p^{\infty})}$ is a finite separable extension.

\begin{definition}\label{def:big.subfields}
Let the {\bf algebraic exponent} of $F/E$,
which we denote by $\mathrm{algex}(F/E)$,
be the maximum of the degrees $[E(a):E]$, for $a\in F$, if this is finite.
If there is no maximum then we simply write $\mathrm{algex}(F/E)=\infty$.
We say that $E$ is a {\bf big subfield} of $F$ if $\mathrm{algex}(F/E)<\infty$.
\end{definition}

We denote by $EF^{(p^{\infty})}$ the compositum of $E$ and $F^{(p^{\infty})}$, taken inside the common extension $F$.

\begin{lemma}\label{lem:big.1}
If $E$ is a big subfield of $F$,
then $EF^{(p^{\infty})}/E$ is a finite separable extension of degree $\mathrm{algex}(EF^{(p^{\infty})}/E)$.
\end{lemma}
\begin{proof}
Since $E$ is a big subfield of $F$, $E$ is also a big subfield of $EF^{(p^{\infty})}$.
In particular, $EF^{(p^{\infty})}/E$ is algebraic.
Let $a\in F^{(p^{\infty})}$ and
suppose that $a$ is not separably algebraic over $E$.
Then $E(a^{p^{-1}})/E(a)$ is a purely inseparable extension of degree $p$.
In fact $E(a^{p^{-k}})/E(a)$ is a purely inseparable extension of degree $p^{k}$, for each $k<\omega$.
Since $a$ is an element of $F^{(p^{\infty})}$, the algebraic exponent of $EF^{(p^{\infty})}/E$ is infinite,
which contradicts our assumption.
Therefore $a$ is separably algebraic over $E$.
Since $a\in F^{(p^{\infty})}$ was arbitrary, this shows that
$EF^{(p^{\infty})}/E$ is a separable algebraic extension.

Let $N$ denote the algebraic exponent of $EF^{(p^{\infty})}/E$.
Note that $N\leq\mathrm{algex}(F/E)<\infty$.
Therefore there exists $b\in EF^{(p^{\infty})}$ such that $N=[E(b):E]$.
We claim that $E(b)=EF^{(p^{\infty})}$.
If not then there exists $c\in EF^{(p^{\infty})}\setminus E(b)$,
whence
$[E(b,c):E]>N$.
Since $E(b,c)/E$ is a finite separable extension, it is generated by a single element, $d$ say.
Therefore the degree $[E(d):E]$ is strictly greater than $N$, which contradicts our choice of $N$.
This shows that $E(b)=EF^{(p^{\infty})}$, as claimed.
In particular, $EF^{(p^{\infty})}/E$ is a finite separable extension of degree $N$.
\end{proof}

\begin{figure}[ht]
\begin{center}
$$\begindc{\commdiag}[250]
\obj(0,3)[E]{$E$}
\obj(3,2)[FP]{$F^{(p^\infty)}$}
\obj(3,4)[EFP]{$EF^{(p^\infty)}$}
\obj(3,6)[F]{$F$}
\mor{E}{EFP}{}[\atleft,\solidline]
\mor{FP}{EFP}{}[\atleft,\solidline]
\mor{EFP}{F}{}[\atleft,\solidline]
\enddc$$
\end{center}
\caption{Illustration of \autoref{lem:big.1}}
\label{fig:1}
\end{figure}

\begin{lemma}\label{fact:min.poly}
Let $F/E$ be a field extension.
Suppose that $a$ is an element in an extension of $F$ which is of degree $n$ both over $F$ and over $E$.
Then the minimal polynomial of $a$ over $F$ is an element of $E[x]$.
\end{lemma}
\begin{proof}
Let $m_{1}\in F[x]$ and $m_{2}\in E[x]$ be the minimal polynomials.
Then $m_{1}$ divides $m_{2}$ in the polynomial ring $F[x]$.
Since they are both monic and of the same degree, they are equal.
\end{proof}

\begin{proposition}\label{prp:big.subfields}
If $E$ is a big subfield of $F$ then $F^{(p^\infty)}/E^{(p^{\infty})}$ is a finite separable extension of degree $[EF^{(p^{\infty})}:E]$.
\end{proposition}
\begin{proof}
By \autoref{lem:big.1}, $EF^{(p^{\infty})}/E$ is a finite separable extension.
Let $N$ denote the degree $[EF^{(p^{\infty})}:E]$.
Since $E$ is a subfield of $F$, $E^{(p^{\infty})}$ is a subfield of $F^{(p^{\infty})}$.
Moreover, $F^{(p^{\infty})}/E^{(p^{\infty})}$ is a separable extension, since $E^{(p^{\infty})}$ is perfect.
Let $a\in F^{(p^\infty)}$.
Since $E$ is a big subfield of $EF^{(p^{\infty})}$,
the degrees
$\big([E(a^{p^{-k}}):E]\big)_{k<\omega}$
are each bounded by $N$.
Let $N_{0}\leq N$ denote the maximum
$$\mathrm{max}\big\{[E(a^{p^{-k}}):E]\;\big|\;k<\omega\big\},$$
and fix $k_{0}<\omega$ such that
$N_{0}=[E(a^{p^{-k_{0}}}):E]$.
Then for each $k<\omega$,
we have
$N_{0}=[E(a^{p^{-k-k_{0}}}):E]$;
and by applying the $k$-fold Frobenius isomorphism,
we obtain
$N_{0}=\big[E^{(p^{k})}(a^{p^{-k_{0}}}):E^{(p^{k})}\big]$.

Let $m\in E[x]$ be the minimal polynomial of $a^{p^{-k_{0}}}$ over $E$.
By \autoref{fact:min.poly} we have
$m\in E^{(p^{k})}[x]$.
Indeed, this holds for all $k<\omega$, and thus we obtain
$m\in E^{(p^{\infty})}[x]$.
Consequently, $m$ is the minimal polynomial of $a^{p^{-k_{0}}}$ over $E^{(p^{\infty})}$.
This shows that $a^{p^{-k_{0}}}$ is separably algebraic of degree $N_{0}$ over $E^{(p^{\infty})}$.
Thus $a$ is also separably algebraic of degree $\leq N_{0}$ over $E^{(p^{\infty})}$.
Since $a\in F^{(p^{\infty})}$ was arbitrary, 
$E^{(p^{\infty})}$ is a big subfield of $F^{(p^{\infty})}$.
Another application of \autoref{lem:big.1} shows that
$F^{(p^{\infty})}/E^{(p^{\infty})}$ is a finite separable extension.
Since $E/E^{(p^{\infty})}$ is a regular extension, it is linearly disjoint from
$F^{(p^{\infty})}/E^{(p^{\infty})}$.
Therefore
$[F^{(p^{\infty})}:E^{(p^{\infty})}]=[EF^{(p^{\infty})}:E]$,
as required.
\end{proof}

\begin{figure}[ht]
\begin{center}
$$\begindc{\commdiag}[250]
\obj(0,3)[E]{$E$}
\obj(3,2)[FP]{$F^{(p^\infty)}$}
\obj(3,4)[EFP]{$EF^{(p^\infty)}$}
\obj(3,6)[F]{$F$}
\obj(0,1)[EP]{$E^{(p^{\infty})}$}
\mor{E}{EFP}{$N$}[\atleft,\solidline]
\mor{FP}{EFP}{}[\atleft,\solidline]
\mor{EFP}{F}{}[\atleft,\solidline]
\mor{EP}{E}{}[\atleft,\solidline]
\mor{EP}{FP}{$N$}[\atright,\solidline]
\enddc$$
\end{center}
\caption{Illustration of \autoref{prp:big.subfields}}
\label{fig:2}
\end{figure}


\section{Uniformly big subfields}
\label{section:uniformly.big}

In this section we introduce `uniformly big subfields'.
eet $F$ be any field and let $\phi(x;\mathbf{q})$ be an $\mathcal{L}_{\mathrm{ring}}$-formula with one free-variable $x$ and a tuple $\mathbf{q}\in F^{\mathbf{y}}$ of parameters.
Denote by $X=\phi(F;\mathbf{q})$ the subset of $F$ defined by $\phi(x;\mathbf{q})$, and let $E=(X)$ be the subfield of $F$ which is generated by $X$.
For an elementary extension $F\preceq F^{*}$, we denote by $X^{*}=\phi(F^{*};\mathbf{q})$ the set defined in $F^{*}$ by $\phi(x;\mathbf{q})$,
and we denote by $E^{(*)}:=(X^{*})$ the subfield of $F^{*}$ generated by $X^{*}$.\footnote{The superscript `$(*)$' is intended to indicate that $E^{(*)}$ depends on $\phi(x;\mathbf{q})$ rather than on $E$.
For example, the subfields generated by the sets defined by the formulas $x=1$ and $x=x$ coincide for $\mathbb{Q}$ but not for $\mathbb{Q}^{*}\succ\mathbb{Q}$.}
Note that neither $E$ nor $E^{(*)}$ is assumed to be definable.
However, each is the union of a chain of definable sets, as follows.

Fix an enumeration $(f_{i})_{i<\omega}$ of all the multivariable polynomials over the prime field.
For convenience, we arrange the enumeration so that $f_{i}$ is a polynomial in (at most) the variables $X_{0},...,X_{i-1}$.
For $m<\omega$,
we let $\phi_{m}(x;\mathbf{q})$ denote the formula
\begin{align*}
\exists\;
\mathbf{a}=(a_{k})_{k<m},\mathbf{b}=(b_{l})_{l<m}\;:\;
\bigvee_{i,j<m}\bigg(
x\cdot f_{i}(\mathbf{a})=f_{j}(\mathbf{b})
\wedge
f_{i}(\mathbf{a})\neq0
\wedge\bigwedge_{k,l<m}\Big(\phi(a_{k};\mathbf{q})\wedge\phi(b_{l};\mathbf{q})\Big)\bigg).
\end{align*}
Note that if $\phi(x;\mathbf{q})$ is an existential formula, then $\phi_{m}(x;\mathbf{q})$ is also (equivalent to) an existential formula, for each $m<\omega$.
In $F$, $\phi_{m}(x;\mathbf{q})$ defines the union of images of $X$ under the rational functions $\frac{f_{i}}{f_{j}}$, for $i,j<m$.
More precisely, we have
\begin{align*}
\phi_{m}(F;\mathbf{q})=\bigg\{\frac{f_{i}(\mathbf{a})}{f_{j}(\mathbf{b})}\;\bigg|\;\mathbf{a},\mathbf{b}\subseteq\phi(F;\mathbf{q}),f_{j}(\mathbf{b})\neq0,\text{ and }i,j<m,\bigg\}.
\end{align*}
We write $X_{m}:=\phi_{m}(F;\mathbf{q})$, and similarly $X_{m}^{*}:=\phi_{m}(F^{*};\mathbf{q})$.
Therefore the field $E=(X)$ is the union $\bigcup_{m<\omega}X_{m}$, and the field $E^{(*)}=(X^{*})$ is the union $\bigcup_{m<\omega}X^{*}_{m}$.

\begin{definition}\label{def:uniformly.big.subfields}
We say that $E=(X)$ is a {\bf uniformly big subfield} of $F$ if $E^{(*)}=(X^{*})$ is a big subfield of $F^{*}$, for all elementary extensions $F\preceq F^{*}$.
\end{definition}

Note that being `uniformly big' is not strictly a property of the subfield $E$, but rather a property of a choice of definable generating set $X$.
This slight ambiguity will not cause a problem.

Let $\mathbf{x}=(x_{1},...,x_{n})$ be an $n$-tuple of variables.
For $m,n<\omega$, we let $\delta_{m,n}(\mathbf{x};\mathbf{q})$ be the formula
\begin{align*}
\exists\;
\mathbf{a}=(a_{i})_{i<m^{n}}\;:\;
\bigg(\bigwedge_{i<m^{n}}\phi_{m}(a_{i};\mathbf{q})
\wedge\neg\bigwedge_{i<m^{n}}a_{i}=0
\wedge\sum_{\sum_{j}i_{j}<m}a_{i}x_{1}^{i_{1}}...x_{n}^{i_{n}}=0\bigg).
\end{align*}
Again, note that if $\phi(x;\mathbf{q})$ is an existential formula, then $\delta_{m,n}(\mathbf{x};\mathbf{q})$ is also (equivalent to) an existential formula, for each $m,n<\omega$.
In $F$, $\delta_{m,n}(\mathbf{x};\mathbf{q})$ defines the set of those $n$-tuples which are zeroes of some nontrivial polynomial of total degree $<m$ with coefficients from $X_{m}$.
The set
$r_{n}(\mathbf{x}):=\{\neg\delta_{m,n}(\mathbf{x};\mathbf{q})\;|\;m<\omega\}$
is a (partial) $n$-type over $\mathbf{q}$.
An $n$-tuple $\mathbf{a}\in F^{n}$ realises $r_{n}(\mathbf{x})$ if and only if $\mathbf{a}$ is algebraically independent over $E=(X)$.


\begin{proposition}\label{prp:uniform.equivalences}
The following are equivalent:
\begin{enumerate}
\item $E=(X)$ is a uniformly big subfield of $F$;
\item there exists $m<\omega$ such that
$F\;\models\;\forall x_{1}\;\delta_{m,1}(x_{1};\mathbf{q})$; and
\item $\mathrm{trdeg}(F/E)$ is `bounded',
i.e.\ there is a cardinal $\kappa$ such that for all elementary extensions $F\preceq F^{*}$ we have
$$\mathrm{trdeg}(F^{*}/E^{(*)})\leq\kappa.$$
\end{enumerate}
\end{proposition}
\begin{proof}
$(1\implies3)$: Suppose that $E=(X)$ is a uniformly big subfield of $F$,
and let $F\preceq F^{*}$ be any elementary extension.
Then $E^{(*)}=(X^{*})$ is a big subfield of $F^{*}$, and in particular $F^{*}/E^{(*)}$ is algberaic,
i.e.\ $\mathrm{trdeg}(F^{*}/E^{(*)})\leq0$.
Thus even $0$ is the required bound.

$(2\implies1)$: Suppose that $F\models\forall x_{1}\;\delta_{m,1}(x_{1})$.
Let $F\preceq F^{*}$ be any elementary extension and let $a\in F^{*}$.
Then $F^{*}\models\delta_{m,1}(a;\mathbf{q})$, so $a$ is the zero of a non-trivial polynomial of degree $<m$ with coefficients from
$X_{m}^{*}=\phi_{m}(F^{*};\mathbf{q})$.
In particular, $a$ is algebraic of degree $<m$ over $E^{(*)}=(X^{*})$.
This shows that
$\mathrm{algex}(F^{*}/E^{(*)})<m$,
and so $E^{(*)}$ is a big subfield of $F^{*}$.
Since $F^{*}$ is an arbitrary elementary extension of $F$, this shows that $E$ is a uniformly big subfield of $F$.

$(3\implies 2)$: in fact we show the contrapositive $(\neg2\implies\neg3)$.
This is a standard compactness argument.
Indeed, by compactness, to show the negation of $(3)$ it suffices to show that $r_{n}(x_{1},...,x_{n})$ is consistent, for each $n<\omega$.
We suppose the negation of $(2)$,
i.e.\ for each $m<\omega$ we have
$F\;\models\;\neg\forall x_{1}\;\delta_{m,1}(x_{1};\mathbf{q})$.
By compactness it follows that $r_{1}(x_{1})$ is consistent.
We proceed by induction, and the consistency of $r_{1}(x_{1})$ is the base case.

Suppose that $r_{n}(x_{1},...,x_{n})$ is consistent.
Then there exists an elementary extension $F\preceq F^{*}$ and an $n$-tuple $(a_{1},...,a_{n})$ in $F^{*}$ that realises
the type $r_{n}(x_{1},...,x_{n})$.
That is, $(a_{1},...,a_{n})$ is algebraically independent over $E^{(*)}$.
Let $m<\omega$.
Trivially, $(a_{1}^{m},...,a_{n}^{m})$ is also algebraically independent over $E^{(*)}$.
Moreover, $a_{n}$ is of degree exactly $m$ over $E^{(*)}(a_{1}^{m},...,a_{n}^{m})$.
In particular $(a_{1}^{m},...,a_{n}^{m},a_{n})$ is not the zero of any nontrivial polynomial of total degree $<m$ with coefficients from $E^{(*)}$.
Since $X_{m}^{*}\subseteq E^{(*)}$, it follows
{\em a fortiori}
that 
$(a_{1}^{m},...,a_{n}^{m},a_{n})$
is not
a zero of any nontrivial polynomial of total degree $<m$ with coefficients from $X_{m}^{*}$;
i.e.\ we have
$$F^{*}\;\models\;\neg\delta_{m,n+1}(a_{1}^{m},...,a_{n}^{m},a_{n};\mathbf{q}).$$
It follows that the type $r_{n+1}(x_{1},...,x_{n+1})$ is consistent, as required.
By induction, $r_{n}(x_{1},...,x_{n})$ is consistent, for all $n<\omega$.
As indicated above, this entails the negation of $(3)$, as required.
\end{proof}

Finally for this section, we give a straightforward lemma that will be used in the proof of \autoref{prp:Egen.uniformly.big}.

\begin{lemma}\label{lem:monotone}
Let $F\preceq F^{*}$ be an elementary extension.
Then $\mathrm{trdeg}(F/E)\leq\mathrm{trdeg}(F^{*}/E^{(*)})$.
\end{lemma}
\begin{proof}
Let $\mathbf{a}\in F^{n}$ be algebraically independent over $E$.
Then we have $F\;\models\;\neg\delta_{m,n}(\mathbf{a};\mathbf{q})$, for all $m<\omega$.
Thus also $F^{*}\;\models\;\neg\delta_{m,n}(\mathbf{a};\mathbf{q})$, for all $m<\omega$;
and therefore $\mathbf{a}$ is algebraically independent over $E^{(*)}$.
\end{proof}



\section{Existentially generated subfields of henselian fields}
\label{section:Egen.subfields}

Let $(K,v)$ be an henselian nontrivially valued field, with value group $\Gamma_{v}$, and let $C\subseteq K$ be a subfield such that $K/C$ is separable.

\subsection{The henselian topology}
\label{section:henselianity}

In \cite{PZ}, Prestel and Ziegler study topological fields.
In particular, they study those topologies induced by nontrivial valuations, as follows.

Let $a\in K$ and let $\alpha\in\Gamma_{v}$.
We define the {\bf ball} of radius $\alpha$ around $a$ to be
$B(\alpha;a):=\{x\in K\;|\;v(x-a)>\alpha\}$.
The collection $\{B(\alpha;a)\;|\;a\in K,\alpha\in\Gamma_{v}\}$ forms a base for the open sets of the {\bf valuation topology} induced by $v$, which is a field topology on $K$.
For tuples $\mathbf{a}=(a_{1},...,a_{n})\in K^{n}$ and $\bfalpha=(\alpha_{1},...,\alpha_{n})\in\Gamma_{v}^{n}$,
we defined the {\bf ball} of radius $\bfalpha$ around $\mathbf{a}$ to be
$$B(\bfalpha;\mathbf{a}):=\prod_{i=1}^{n}B(\alpha_{i};a_{i}).$$

In Section 7 of \cite{PZ}, Prestel and Ziegler study topological consequences of henselianity.
Specifically, they shows that
`topological henselianity'
is equivalent to the topology satisfying the `Implicit Function Theorem', in the context of the continuity of functions defined implicitly by the vanishing of polynomials.
We refer the reader to \cite{PZ} for more details.
For the present, we simply give the following fact.

\begin{fact}[{cf \cite[(7.4) Theorem]{PZ}}]\label{fact:IFT}
For all $f\in K[X_{1},...,X_{n},Y]$ and all $(\mathbf{a},b)\in K^{n+1}$,
if $f(\mathbf{a},b)=0$ and $f'(\mathbf{a},b)\neq0$,
then there exist $(\bfalpha,\beta)\in\Gamma_{v}^{n+1}$ such that
for all $\mathbf{a}'\in B(\bfalpha;\mathbf{a})$ there exists a unique $b'\in B(\beta;b)$ such that
$f(\mathbf{a},b)=0$.
Moreover, the map
$\mathbf{a}\longmapsto b'$
is continuous.
\end{fact}
\begin{proof}
This follows from \cite[(7.4) Theorem]{PZ} since the topology induced by the a henselian valuation is indeed a henselian topology.
\end{proof}

In \cite[Section 4]{K2}, Kuhlmann presents a multidimensional version of the Implicit Function Theorem, which we rewrite very slightly for our convenience.

\begin{fact}[{cf \cite[Section 4]{K2}}]\label{fact:MIFT}
Let $(K,v)$ be a henselian valued field.
Let $f_{1},...,f_{n}\in K[X_{1},...,X_{l}]$ with $n<l$.
Set
$$\tilde{J}:=
\begin{pmatrix}
\frac{\partial f_{1}}{\partial X_{l-n+1}}&\ldots&\frac{\partial f_{1}}{\partial X_{l}}\\
\vdots& & \vdots\\
\frac{\partial f_{n}}{\partial X_{l-n+1}}&\ldots&\frac{\partial f_{n}}{\partial X_{l}}
\end{pmatrix}.
$$
Assume that $f_{1},...,f_{n}$ admit a common zero $\mathbf{a}=(a_{1},...,a_{l})\in K^{l}$
and that $\det\tilde{J}(\mathbf{a})\neq0$.
Then there is some $\alpha\in\Gamma_{v}$ such that for all $(a_{1}',...,a_{l}')\in K^{l-n}$
with $v(a_{i}-a_{i}')>2\alpha$, for $i\in\{1,...,l-n\}$,
there exists a unique $(a_{l-n+1}',...,a_{l}')\in K^{n}$ such that
$(a_{1}',...,a_{l}')$ is a common zero of $f_{1},...,f_{n}$,
and $v(a_{i}-a_{i}')>\alpha$, for $i\in\{l-n+1,...,l\}$.
\end{fact}


\begin{lemma}\label{lem:separable.projection}
Let $(\mathbf{c},\mathbf{d})\in K^{m+n}$ and suppose that 
$\mathbf{c}$ is a separating transcendence base for
$C(\mathbf{c},\mathbf{d})/C$.
Then there exists $(\bfgamma,\bfdelta)\in\Gamma_{v}^{m+n}$ such that
\begin{align*}
\locus(\mathbf{c},\mathbf{d}/C)\cap B(\bfgamma,\bfdelta;\mathbf{c},\mathbf{d})
\end{align*}
is the graph of a continuous function
\begin{align*}
\mathbf{f}:B(\bfgamma;\mathbf{c})&\longrightarrow B(\bfdelta;\mathbf{d}).
\end{align*}
\end{lemma}
\begin{proof}
%
Write $\mathbf{d}=(d_{1},...,d_{n})$ and let $j\in\{1,...,n\}$.
We let $g_{j}\in C[\mathbf{X},Y_{1},...,Y_{j}]$ and $h_{j}\in C[\mathbf{X},Y_{1},...,Y_{j-1}]$ be polynomials such that
$h_{j}(\mathbf{c},d_{1},...,d_{j-1})\neq0$ and
$$\frac{g_{j}(\mathbf{c},d_{1},...,d_{j-1},Y_{j})}{h_{j}(\mathbf{c},d_{1},...,d_{j-1})}$$
is the minimal polynomial of $d_{j}$ over $C(\mathbf{c},d_{1},...,d_{j-1})$.
Since $\mathbf{c}$ is a separating transcendence base for $C(\mathbf{c},\mathbf{d})/C$, $d_{j}$ is separably algebraic over $C(\mathbf{c},d_{1},...,d_{j-1})$, and so
$g_{j}(\mathbf{c},d_{1},...,d_{j})=0$ and
$$\frac{\partial g_{j}}{\partial Y_{j}}(\mathbf{c},d_{1},...,d_{j})\neq0.$$
Let $Z(g_{1},...,g_{n})$ denote the common zeroes of $g_{1},...,g_{n}$ in $K^{m+n}$, this is simply the set of $K$-rational points of the algebraic set defined in affine $m+n$-space by the vanishing of $g_{1},...,g_{n}$.
By our definition of the locus, we have $\locus(\mathbf{c},\mathbf{d}/C)\subseteq Z(g_{1},...,g_{n})$.
In fact, there is a Zariski open set $U$ such that $(\mathbf{c},\mathbf{d})\in U$ and 
\begin{align*}
Z(g_{1},...,g_{n})\cap U&=\locus(\mathbf{c},\mathbf{d}/C)\cap U.
\end{align*}
If we define $\tilde{J}$ as above, using the polynomials $g_{1},...,g_{n}$,
then $\tilde{J}(\mathbf{c},\mathbf{d})$ is a lower triangular matrix with non-zero entries on the diagonal.
Therefore $\det\tilde{J}(\mathbf{c},\mathbf{d})\neq0$.
By applying \autoref{fact:MIFT}
we obtain $\alpha\in\Gamma_{v}$ such that,
by writing $\bfgamma=(2\alpha,...,2\alpha)$ and $\bfdelta=(\alpha,...,\alpha)$,
we find that
\begin{align*}
Z(g_{1},...,g_{n})\cap B(\bfgamma,\bfdelta;\mathbf{c},\mathbf{d})
\end{align*}
is the graph of a continuous function
\begin{align*}
B(\bfgamma;\mathbf{c})&\longrightarrow B(\bfdelta;\mathbf{d}).
\end{align*}
By the continuity, we may assume that $B(\bfgamma,\bfdelta;\mathbf{c},\mathbf{d})\subseteq U$, and so we are done.
\end{proof}

\begin{lemma}\label{lem:x-projection}
Let $(c,\mathbf{d})\in K^{1+n}$ and suppose that 
$c$ is transcendental over $C$,
and $C(c,\mathbf{d})/C(c)$ is separable.
Then there exists $\gamma\in\Gamma_{v}$ such that
\begin{align*}
B(\gamma;c)&\subseteq\pr_{x}\locus(c,\mathbf{d}/C),
\end{align*}
where $\pr_{x}:K^{1+n}\longrightarrow K$ is the projection onto the $x$-coordinate, which is the first coordinate.
\end{lemma}
\begin{proof}
Re-ordering if necessary, write $\mathbf{d}$ as the pair $(\mathbf{d}_{1},\mathbf{d}_{2})$, where $\mathbf{d}_{1}\subseteq\mathbf{d}$ is a separating transcendence base for $C(c,\mathbf{d})/C(c)$,
Then $(c,\mathbf{d}_{1})$ is a separating transcendence base for $C(c,\mathbf{d})/C$.
Suppose that $\mathbf{d}_{1}$ is an $n_{1}$-tuple.
By \autoref{lem:separable.projection}, there exists $(\gamma,\bfdelta_{1})\in\Gamma_{v}^{1+n_{1}}$
such that
$$B(\gamma,\bfdelta_{1};c,\mathbf{d}_{1})\subseteq\pr_{x,\mathbf{y}_{1}}\locus(c,\mathbf{d}/C),$$
where $\pr_{x,\mathbf{y}_{1}}$ is the map projecting a $(1+n)$-tuple onto its first $(1+n_{1})$-many coordinates.
Projecting again, onto the $x$-coordinate, we obtain $B(\gamma;c)\subseteq\pr_{x}\locus(c,\mathbf{d}/C)$,
as required.
\end{proof}

\subsection{Bounding the transcendence degree}
\label{subsection:bounding.trdeg}


We suppose throughout this subsection that $K$ is $\aleph_{1}$-saturated.\footnote{In fact it suffices that $K$ is $\aleph_{0}$-saturated. To see this, one argues that $\Lambda_{K}(\mathbf{c})$ is in the definable closure of $\mathbf{c}$. Rather than give the details here, we simply assume that $K$ is $\aleph_{1}$-saturated.}
This subsection is devoted to a proof of
\autoref{prp:Egen.bounded}.

Just as in \autoref{section:uniformly.big},
we let $\phi(x;\mathbf{q})$ be an {\em existential} $\mathcal{L}_{\mathrm{ring}}$-formula with one free-variable $x$ and a tuple $\mathbf{q}\in K^{\mathbf{y}}$ of parameters.
We denote by $X=\phi(K;\mathbf{q})$ the subset defined by $\phi(x;\mathbf{q})$.
Let $\Lambda_{K}(\mathbf{q})$ denote the relative inseparable closure in $K$ of the subfield generated by the parameters $\mathbf{q}$.
Let $C$ be the relative algebraic closure in $K$ of $\Lambda_{K}(\mathbf{q})$.
Note that $K/C$ is regular.

Existentially definable sets, such as $X$, are the projections of algebraic sets.
More precisely, using a standard reduction, there exist finitely many multivariable polynomials $f_{1},...,f_{k}\in C[x,y_{1},...,y_{r}]$ such that
$X=\pr_{x}(V)$,
where
$$V=\big\{(x,y_{1},...,y_{r})\in K^{1+r}\;\big|\;\text{$f_{i}(x,y_{1},...,y_{k})=0$, for all $i\leq k$}\big\}$$
is the set of $(1+r)$-tuples from $K$ at which each of the polynomials $f_{i}$ vanishes,
and
$$\pr_{x}: K^{1+r}\longrightarrow K$$
is the map that projects an $r$-tuple onto its $x$-coordinate.
In other language, $V$ is the set of $K$-rational points of the algebraic set defined to be the common zero-locus of the polynomials $f_{i}$, for $i\leq k$.

By our assumption that $K$ is $\aleph_{1}$-saturated, and by \autoref{lem:countable},
we may realise in $K$ the type of an element of $X$ which is transcendental over $C$.
Let $a\in K$ denote such an element, and let $\mathbf{b}\in K^{r}$ be an $r$-tuple such that $(a,\mathbf{b})\in V$.
%
Since $K/C$ is separable, the tuple $(a,\mathbf{b})$ may be reordered into a tuple
$(\mathbf{c},\mathbf{d})$ such that $\mathbf{c}$ is a separating transcendence base of $C(a,\mathbf{b})/C$.
That is, $\mathbf{c}$ is algebraically independent over $C$,
and the extension $C(\mathbf{c},\mathbf{d})/C(\mathbf{c})$ is separably algebraic.
In particular, note that $a$ is separably algebraic over $C(\mathbf{c})$.
Also, since $a$ is transcendental over $C$, $\mathbf{c}$ is not the empty tuple.

Since $(\mathbf{c},\mathbf{d})$ is just a reordering of $(a,\mathbf{b})$, we have $\pr_{x}\locus(\mathbf{c},\mathbf{d}/C)\subseteq\pr_{x}V=X$, where $\pr_{x}$ still denotes the projection onto the $x$-coordinate,  corresponding to $a$, even if this is no longer the first coordinate.
Write $1+r=s+t$, where $\mathbf{c}$ is an $s$-tuple and $\mathbf{d}$ is a $t$-tuple.
%
By \autoref{lem:separable.projection}, there exists $(\bfgamma,\bfdelta)\in\Gamma_{v}^{s+t}$ and a continuous function
$\mathbf{f}:B(\bfgamma;\mathbf{c})\longrightarrow B(\bfdelta;\mathbf{d})$
such that
\begin{align*}
\mathrm{graph}(\mathbf{f})=\locus(\mathbf{c},\mathbf{d}/C)\cap B(\bfgamma,\bfdelta;\mathbf{c},\mathbf{d}).
\end{align*}
Therefore
$\pr_{x}\big(\mathrm{graph}(\mathbf{f})\big)\subseteq X$.

\begin{lemma}\label{lem:isomorphic}
Let $F/C$ be a field extension and let $\mathbf{a}=(a_{1},...,a_{n})$ and $\mathbf{a}'=(a_{1}',...,a_{n}')$ be two $n$-tuples from $F$.
Suppose that
$\mathbf{a}'\in\locus(\mathbf{a}/C)$ and
$\mathrm{trdeg}(\mathbf{a}/C)\leq\mathrm{trdeg}(\mathbf{a}'/C)$.
Then there is a $C$-isomorphism $\rho:C(\mathbf{a})\longrightarrow C(\mathbf{a}')$ that maps $a_{i}\longmapsto a_{i}'$, for $i\leq n$.
\end{lemma}
\begin{proof}
This is just a rewriting of \cite[Theorem 5]{LangAG} into our language.
\end{proof}

\begin{lemma}\label{lem:exchange}
There exists $c_{1}\in\mathbf{c}$ such that $c_{1}$ is algebraic over $C(\mathbf{c}_{2},a)$, where $\mathbf{c}_{2}:=\mathbf{c}\setminus\{c_{1}\}$.
\end{lemma}
\begin{proof}
Our assumptions are that $\mathbf{c}$ is algebraically independent over $C$, $a$ is transcendental over $C$, and $a$ is algebraic over $C(\mathbf{c})$.
Therefore, we may choose a subtuple $\mathbf{c}_{2}\subset\mathbf{c}$ of length $n-1$ such that $(\mathbf{c}_{2},a)$ is algebraically independent over $C$.
Let $c_{1}$ denote the unique element of $\mathbf{c}\setminus\mathbf{c}_{2}$.
Clearly $c_{1}$ is algebraic over $C(\mathbf{c}_{2},a)$, by the Exchange Property.
\end{proof}

By reordering if necessary, we write $\mathbf{c}=(c_{1},\mathbf{c}_{2})$.
Fix the corresponding reordering $(\gamma_{1},\bfgamma_{2})$ of $\bfgamma$.
We reiterate that $\pr_{x}$ still denotes the projection onto the $x$-coordinate, corresponding to $a$, so that $\pr_{x}(c_{1},\mathbf{c}_{2},\mathbf{d})=a$.

\begin{lemma}\label{lem:local.algebraic}
Let $e\in B(\gamma_{1};c_{1})$ and let
$\alpha$ be the $x$-coordiante of $(e,\mathbf{c}_{2},\mathbf{f}(e,\mathbf{c}_{2}))$,
so that $\alpha=\pr_{x}(e,\mathbf{c}_{2},\mathbf{f}(e,\mathbf{c}_{2}))$.
Then $e$ is algebraic over $C(\mathbf{c}_{2},\alpha)$.
\end{lemma}
\begin{proof}
If $e$ is algebraic over $C(\mathbf{c}_{2})$ then it is certainly algebraic over $C(\mathbf{c}_{2},\alpha)$, as required.

On the other hand, suppose that $e$ is transcendental over $C(\mathbf{c}_{2})$.
Since $\mathbf{d}$ is algebraic over $C(\mathbf{c})$,
it follows that
$\mathrm{trdeg}(\mathbf{c},\mathbf{d}/C)\leq\mathrm{trdeg}(e,\mathbf{c}_{2},\mathbf{f}(e,\mathbf{c}_{2})/C)$.
Since $a\in B(\gamma_{1};c_{1})$,
$\mathbf{f}$ is well-defined on the pair $(e,\mathbf{c}_{2})$,
and we have
$(e,\mathbf{c}_{2},\mathbf{f}(e,\mathbf{c}_{2}))\in\mathrm{locus}(\mathbf{c},\mathbf{d}/C)$.
Putting these together, we have satisfied the hypotheses of \autoref{lem:isomorphic}, using which we obtain the $C$-isomorphism
$$\rho:C(\mathbf{c},\mathbf{d})\longrightarrow C(e,\mathbf{c}_{2},\mathbf{f}(e,\mathbf{c}_{2})).$$
Moreover, applying $\rho$ coordinate-wise, we have $\rho(\mathbf{c},\mathbf{d})=(e,\mathbf{c},\mathbf{f}(e,\mathbf{c}_{2}))$.
Thus $\rho(c_{1})=e$ and $\rho(a)=\alpha$.
The result now follows from \autoref{lem:exchange}.
\end{proof}

Let $C(\mathbf{c}_{2},X)$ denote the subfield of $K$ generated over $C(\mathbf{c}_{2})$ by all elements of $X$.

\begin{proposition}\label{prp:Egen.bounded}
Every element of $K$ is algebraic over $C(\mathbf{c}_{2},X)$.
\end{proposition}
\begin{proof}
It follows from \autoref{lem:local.algebraic} that
every element of $B(\gamma_{1};c_{1})$ is algebraic over
$C(\mathbf{c}_{2},X)$.
In particular, $c_{1}$ is algebraic over $C(\mathbf{c}_{2},X)$.
Therefore every element of $B(\gamma_{1};0)$ is algebraic over $C(\mathbf{c}_{2},X)$.

Let $t\in K$.
If $t\in B(\gamma_{1};0)$, then $t$ is algebraic over $E$.
Otherwise, suppose that $t\notin B(\gamma_{1};0)$.
Choose any $s\in B(\gamma_{1};0)\setminus\{0\}$ and note that $s$ is algebraic over $C(\mathbf{c}_{2},X)$.
Then $v(t)<v(s)$, and so $v(s^{-1})<v(t^{-1})$.
As a standard application of the ultrametric triangle inequality, we deduce the equality
$v(s^{-1}+t^{-1})=v(s^{-1})$.
Therefore $(s^{-1}+t^{-1})^{-1}\in B(\gamma_{1};0)$,
and so $(s^{-1}+t^{-1})^{-1}$ is algebraic over $C(\mathbf{c}_{2},X)$.
It follows that $t$ is algebraic over $C(\mathbf{c}_{2},X)$, as required.
\end{proof}


\subsection{The proof of Theorem \ref{thm:main.theorem} for henselian fields}
\label{subsection:proof.henselian}

We are almost in a position to deduce the main theorem, at least in the context of henselian fields.
Note that we no longer assume $K$ to be $\aleph_{1}$-saturated.

\begin{proposition}\label{prp:Egen.uniformly.big}
Let $K$ be a henselian field.
Let $E\subseteq K$ be an existentially generated subfield of $K$.
Then $E$ is a uniformly big subfield of $K$.
\end{proposition}
\begin{proof}
Let $X\subseteq K$ be an infinite subset of $K$ which generates $E$ and is defined by an existential $\mathcal{L}_{\mathrm{ring}}$-formula $\phi(x;\mathbf{q})$.
Let $C$ be the relative algebraic closure in $K$ of $\Lambda_{K}(\mathbf{q})$.
Let $K^{*}\succeq K$ be an arbitrary elementary extension, and let $K^{**}\succeq K^{*}$ be a further $\aleph_{1}$-saturated elementary extension.
As usual, let $X^{*}$ denote the subset of $K^{*}$ defined by $\phi(x;\mathbf{q})$, and write $E^{(*)}=(X^{*})$.
Also, let $X^{**}$ denote the subset of $K^{**}$ defined by $\phi(x;\mathbf{q})$, and write $E^{(**)}=(X^{**})$.
Note that $C$ is the relative algebraic closure of the relative inseparable closure of $\mathbf{q}$ in each of $K$, $K^{*}$, and $K^{**}$.

By $\aleph_{1}$-saturation, we may apply the machinery of the previous section to $K^{**}$.
We choose the tuple $\mathbf{c}_{2}\subseteq K^{**}$ as above.
By \autoref{prp:Egen.bounded}, $K^{**}$ is algebraic over $C(\mathbf{c}_{2},X^{**})$.
Thus we have
\begin{align*}
\mathrm{trdeg}(K^{**}/E^{(**)})&\leq\mathrm{trdeg}(C(\mathbf{c}_{2})),
\end{align*}
where the right hand side is the transcendence degree over the prime field.

By \autoref{lem:monotone}, we have 
$\mathrm{trdeg}(K^{*}/E^{(*)})\leq\mathrm{trdeg}(K^{**}/E^{(**)})$.
Combining this with the previous inequality, we have satisfied clause $(3)$ of \autoref{prp:uniform.equivalences}.
Therefore $E=(X)$ is a uniformly big subfield of $K$.
\end{proof}

Let $\Phi:x\longmapsto x^{p}$ denote the Frobenius map.
Before proving the main proposition of this section, we give four brief lemmas about the interaction of the Frobenius map with balls in the valuation topology and with loci.

\begin{lemma}\label{lem:Frobenius.balls}
Let $\alpha\in\Gamma_{v}$ and $a\in K$.
Then $\Phi\big(B(\alpha;a)\big)=B(p\alpha;a^{p})\cap K^{(p)}$.
\end{lemma}
\begin{proof}
This follows from the fact that $v(a-b)>\alpha$ if and only if $v(a^{p}-b^{p})>p\alpha$.
\end{proof}

\begin{lemma}\label{lem:balls.fields}
Let $n<\omega$, let $a\in K^{(p^{n})}$, and let $\alpha\in\Gamma_{v}$.
Then
\begin{align*}
K^{(p^{n})}&=\bigg\{\frac{x-a}{y-a}\;\bigg|\;x,y\in B(\alpha;a)\cap K^{(p^{n})},y\neq a\bigg\}.
\end{align*}
\end{lemma}
\begin{proof}
Choose $\beta\in\Gamma_{v}$ such that $\alpha\leq p^{n}\beta$.
Let $z\in K$.
Choose $\gamma\in \Gamma_{v}$ with $\gamma\leq v(z)$,
and choose any $y\in K^{\times}$ with $v(y)>\max\{\beta,\beta-\gamma\}$.
Since $v(y)>\beta$, $y\in B(\beta;0)$.
Also
\begin{align*}
v(zy)&=v(z)+v(y)\\
&>\gamma+\beta-\gamma\\
&=\alpha.
\end{align*}
Therefore $zy\in B(\beta;0)$.
This shows that
$K=\{xy^{-1}\;|\;x,y\in B(\beta;0),y\neq 0\}$.
Since $B(\beta;0)=\{x-a^{p^{-n}}\;|\;x\in B(\beta;a^{p^{-n}})\}$, we have
$$K=\{(x-a^{p^{-n}})(y-a^{p^{-n}})^{-1}\;|\;x,y\in B(\beta;a^{p^{-n}}),y\neq a^{p^{-n}}\}.$$
Finally, applying the $n$-th power of Frobenius, we have
$$K^{(p^{n})}=\{(x-a)(y-a)^{-1}\;|\;x,y\in B(p^{n}\beta;a)\cap K^{(p^{n})},y\neq a\},$$
by \autoref{lem:Frobenius.balls}, as required.
\end{proof}

\begin{lemma}\label{lem:Frobenius.locus.i}
Let $(a,\mathbf{b}),(c,\mathbf{d})\in K^{1+n}$.
Then $(c,\mathbf{d})\in\locus(a,\mathbf{b}/C)$ if and only if $(c^{p},\mathbf{d})\in\locus(a^{p},\mathbf{b}/C)$.
\end{lemma}
\begin{proof}
Suppose first that $(c,\mathbf{d})\in\locus(a,\mathbf{b}/C)$.
Let $f\in C[X,\mathbf{Y}]$ be such that $f(a^{p},\mathbf{b})=0$.
Choose $g\in C[X,\mathbf{Y}]$ such that $g(X,\mathbf{Y})=f(X^{p},\mathbf{Y})$.
Then $g(a,\mathbf{b})=0$, and so $g(c,\mathbf{d})=0$, i.e.\ $f(c^{p},\mathbf{d})=0$.
Therefore $(c^{p},\mathbf{d})\in\locus(a^{p},\mathbf{b}/C)$.

For the converse, suppose that $(c^{p},\mathbf{d})\in\locus(a^{p},\mathbf{b}/C)$.
Let $f\in C[X,\mathbf{Y}]$ be such that $f(a,\mathbf{b})=0$.
Denote by $f^{[p]}\in C[X,\mathbf{Y}]$ the polynomial obtained by raising the coefficients (but not the variables) of $f$ to the $p$-th power.
Since the Frobenius map $\Phi$ is a ring homomorphism, we have $\Phi\big(f(X,\mathbf{Y})\big)=f^{[p]}(X^{p},\mathbf{Y}^{p})$, where if $\mathbf{Y}=(Y_{1},... ,Y_{n})$ then $\mathbf{Y}^{p}=(Y_{1}^{p},...,Y_{n}^{p})$.
Therefore $f^{[p]}(a^{p},\mathbf{b}^{p})=0$,
and so $f^{[p]}(c^{p},\mathbf{d}^{p})=0$.
It follows that $f(c,\mathbf{d})=0$.
This shows that $(c,\mathbf{d})\in\locus(a,\mathbf{b}/C)$, as required.
\end{proof}

\begin{lemma}\label{lem:Frobenius.locus}
Let $(a,\mathbf{b})\in K^{x,\mathbf{y}}$.
Then $\Phi\big(\pr_{x}\locus(a,\mathbf{b}/C)\big)=\pr_{x}\locus(a^{p},\mathbf{b}/C)\cap K^{(p)}$.
\end{lemma}
\begin{proof}
This follows from \autoref{lem:Frobenius.locus.i}.
\end{proof}

\begin{proposition}\label{prp:henselian}
Let $K$ be a henselian field and let $\phi(x;\mathbf{q})$ be an existential formula in the language of rings that defines in $K$ an infinite set $X=\phi(K;\mathbf{q})$.
Then there exist $m,n<\omega$ such that
$$K^{(p^{n})}\subseteq(X)_{m}.$$
\end{proposition}
\begin{proof}
As usual, let $C$ denote the relative algebraic closure in $K$ of $\Lambda_{K}(\mathbf{q})$.
Note that $K/C$ is regular.
Let $K^{*}$ be an $\aleph_{0}$-saturated elementary extension of $K$.
Note that $K^{*}/C$ is also regular.
By the saturation of $K^{*}$, we have that
$(K^{*})^{(p^{\infty})}C/C$ is a transcendental extension.
By \autoref{prp:Egen.uniformly.big}, $E:=(X)$ is a uniformly big subfield of $K$,
and thus $E^{(*)}=(X^{*})$ is a big subfield of $K^{*}$.
From \autoref{prp:big.subfields}, we have that $(K^{*})^{(p^{\infty})}/(E^{(*)})^{(p^{\infty})}$ is a finite separable extension.

Therefore $(E^{(*)})^{(p^{\infty})}C/C$ is a transcendental extension.
In particular, there exists
$a\in(E^{(*)})^{(p^{\infty})}$ which is transcendental over $C$.
Let $l<\omega$ be such that $a\in(X^{*})_{l}$.
Since $(X)_{l}$ is also a set defined by an existential $\mathcal{L}_{\mathrm{ring}}$-formula, namely $\phi_{l}(x;\mathbf{q})$, there exists an algebraic set $V_{l}$ such that $(X)_{l}=\pr_{x}(V_{l})$.
Let $\mathbf{b}\in(K^{*})^{\mathbf{y}}$ be such that $(a,\mathbf{b})\in V^{*}_{l}$.
By \autoref{lem:roots}, $K^{*}/C(a^{p^{-n}}|n<\omega)$ is separable.
Since $\mathbf{b}$ is a finite set, there exists $n<\omega$ such that $C(a^{p^{-n}},\mathbf{b})/C(a^{p^{-n}})$ is separable.
By \autoref{lem:x-projection}, there exists $\alpha'\in\Gamma^{*}_{v}$ such that
\begin{align*}
B^{*}(\alpha';a^{p^{-n}})&\subseteq\pr_{x}\locus^{*}(a^{p^{-n}},\mathbf{b}/C),
\end{align*}
where $B^{*}(\alpha';a^{p^{-n}})$ and $\locus^{*}(a^{p^{-n}},\mathbf{b}/C)$ denote the appropriate ball and locus in $K^{*}$.
By iteratively applying the Frobenius map $\Phi$ to both sides of the above inclusion, and setting $\alpha:=p^{n}\alpha'$, we have
\begin{align*}
B^{*}(\alpha;a)\cap(K^{*})^{(p^{n})}&\subseteq\pr_{x}\locus^{*}(a,\mathbf{b}/C),
\end{align*}
by \autoref{lem:Frobenius.balls} and \autoref{lem:Frobenius.locus}.
Note that we have the inclusion $\pr_{x}\locus^{*}(a,\mathbf{b}/C)\subseteq(X^{*})_{l}$.

Applying \autoref{lem:balls.fields} in $K^{*}$, we have
\begin{align*}
(K^{*})^{(p^{n})}&=\bigg\{\frac{x-a}{y-a}\;\bigg|\;x,y\in B^{*}(\alpha;a)\cap (K^{*})^{(p^{n})},y\neq a\bigg\}\\
&\subseteq\bigg\{\frac{x-w}{y-z}\;\bigg|\;w,x,y,z\in(X^{*})_{l},y\neq z\bigg\}. 
\end{align*}
This shows that $(K^{*})^{(p^{n})}$ is contained in the image of $(X^{*})_{l}$ under a certain rational function.
By definition, $(X^{*})_{l}$ is already the union of the image of $X^{*}$ under the rational functions $\frac{f_{i}}{f_{j}}$, for $i,j<l$.\footnote{
In fact, we may combine finitely many rational functions into one.
More precisely, there exist multivariable polynomials $f,g$ over the prime field such that $(X^{*})_{l}=\{f(\mathbf{a})/g(\mathbf{b})\;|\;\mathbf{a},\mathbf{b}\in (X^{*})^{k},g(\mathbf{b})\neq0\}$.
}
Thus there exists $m\geq l$ such that
$(K^{*})^{(p^{n})}\subseteq(X^{*})_{m}$.
Finally, since $K\preceq K^{*}$ is an elementary extension, we have $K^{(p^{n})}\subseteq (X)_{m}$, as required.
\end{proof}


\section{Proof of Theorem \ref{thm:main.theorem}}
\label{section:proof}

We recall one of the several characterizations of largeness given by Pop.

\begin{fact}[{cf \cite[Proposition 1.1]{Pop96}}]
A field $L$ is large if and only if $L$ is existentially closed in $L((t))$.
\end{fact}

Using this characterization, we are able to generalise \autoref{prp:henselian}.


\begin{proposition}\label{prp:large}
Let $L$ be a large field and let $\phi(x;\mathbf{q})$ be an existential formula in the language of rings that defines in $L$ an infinite set $X=\phi(L;\mathbf{q})$.
Then there exist $m,n<\omega$ such that
$$L^{(p^{n})}\subseteq(X)_{m}.$$
\end{proposition}
\begin{proof}
Let $X':=\phi(L((t));\mathbf{q})$ denote the set defined by the same formula in the field $L((t))$.
Note that the $t$-adic valuation on $L((t))$ is both nontrivial and henselian.
Since existential formulas `go up', we have $X\subseteq X'$.
Thus $X'$ is an infinite existentially definable subset of the henselian field $L((t))$.
By \autoref{prp:henselian}, there exist $m,n<\omega$ such that
$$L((t))^{(p^{n})}\subseteq(X')_{m}.$$
Since $L\preceq_{\exists}L((t))$,
we have
$(X)_{m}=(X')_{m}\cap L$.
Therefore we have the inclusions
\begin{align*}
L^{(p^{n})}=L((t))^{(p^{n})}\cap L\subseteq(X')_{m}\cap L=(X)_{m},
\end{align*}
as required.
\end{proof}

Finally, \autoref{thm:main.theorem} is an immediate corollary of \autoref{prp:large}.


\section*{Acknowledgements}
This research forms part of the author's doctoral thesis
which was completed under the supervision of Jochen Koenigsmann and supported by EPSRC.
The author would like to thank Arno Fehm
and Jochen Koenigsmann
for many helpful conversations, feedback, and encouragement.


\bibliographystyle{plain}

\end{document}

%% file: preamble.tex
\usepackage{amsmath}
\usepackage{amssymb}
\usepackage{amsfonts}
\usepackage{amsthm}
\usepackage[margin=2cm]{geometry}
\usepackage{hyperref}
\usepackage{fixmath}
\usepackage{aliascnt}
\usepackage{comment}
\usepackage{pictexwd,dcpic}
\usepackage[usenames,dvipsnames]{color}
\usepackage{graphicx}
\usepackage{stmaryrd}

\theoremstyle{plain}
\newtheorem{theorem}{Theorem}

\newaliascnt{proposition}{theorem}
\newaliascnt{lemma}{theorem}
\newaliascnt{corollary}{theorem}
\newaliascnt{fact}{theorem}
\newaliascnt{observation}{theorem}
\newaliascnt{conjecture}{theorem}
\newaliascnt{definition}{theorem}
\newaliascnt{example}{theorem}
\newaliascnt{question}{theorem}
\newaliascnt{remark}{theorem}
\newaliascnt{property}{theorem}
\newaliascnt{construction}{theorem}
\newaliascnt{setting}{theorem}

\theoremstyle{plain}
\newtheorem{proposition}[proposition]{Proposition}
\newtheorem{lemma}[lemma]{Lemma}

\newtheorem{fact}[fact]{Fact}

\theoremstyle{definition}
\newtheorem{definition}[definition]{Definition}

\theoremstyle{remark}

\aliascntresetthe{proposition}
\aliascntresetthe{lemma}
\aliascntresetthe{corollary}
\aliascntresetthe{fact}
\aliascntresetthe{observation}
\aliascntresetthe{conjecture}
\aliascntresetthe{definition}
\aliascntresetthe{example}
\aliascntresetthe{question}
\aliascntresetthe{remark}
\aliascntresetthe{property}
\aliascntresetthe{construction}
\aliascntresetthe{setting}


%% file: bold.greek.tex
\providecommand{\bfalpha}{\mathbold{\alpha}}

\providecommand{\bfgamma}{\mathbold{\gamma}}

\providecommand{\bfdelta}{\mathbold{\delta}}